\documentclass[12pt, reqno]{amsart}

\allowdisplaybreaks[1]
\usepackage{enumerate}
\usepackage{amsmath}
\usepackage{amssymb}
\usepackage{amsfonts}
\usepackage{lmodern}
\usepackage{verbatim,xspace}
\usepackage[usenames]{color}
\usepackage{tikz-cd}
\usepackage{fullpage}
\usepackage[T1]{fontenc}

\usepackage[english]{babel}

\makeindex

\usepackage[pagebackref,colorlinks,linkcolor=black,citecolor=black,urlcolor=red,hypertexnames=true]{hyperref}

 \newtheorem{theorem}{Theorem}[section]

 \newtheorem{corollary}[theorem]{Corollary}

 \newtheorem{lemma}[theorem]{Lemma}

 \newtheorem{proposition}[theorem]{Proposition}

\theoremstyle{definition}

\newtheorem{example}[theorem]{Example}

\theoremstyle{remark}

\newtheorem{remark}[theorem]{Remark}

\newtheorem{fact*}{Fact}

\newcommand{\N}{\mathbb{N}}

\newcommand{\C}{\mathbb{C}}

\newcommand{\FF}{\mathbb{F}}

\newcommand{\inv}{^{-1}}

\DeclareMathOperator{\Lin}{Lin}

\DeclareMathOperator{\Aut}{Aut}
\DeclareMathOperator{\id}{Id}

\newcommand\al{\alpha}

\newcommand\la{\lambda}

\newcommand{\tphi}{\tilde{\phi}}

\newcommand\bbm{\begin{bmatrix}}

\newcommand\ebm{\end{bmatrix}}

\newcommand{\bpm}{\left( \begin{smallmatrix}}

\newcommand{\epm}{\end{smallmatrix} \right)}

\numberwithin{equation}{section}

\makeatletter

\def\moverlay{\mathpalette\mov@rlay}

\def\mov@rlay#1#2{\leavevmode\vtop{
		
		\baselineskip\z@skip \lineskiplimit-\maxdimen
		
		\ialign{\hfil$#1##$\hfil\cr#2\crcr}}}

\makeatother

\newcommand{\plangle}{\moverlay{(\cr<}}

\newcommand{\prangle}{\moverlay{)\cr>}}

\newcommand{\nfpolys}[3]{\mathbb{#1}{<}{#2}_1,\ldots, {#2}_{#3}{>}}

\newcommand{\nfrats}[3]{#1\plangle{#2}_1,\ldots, {#2}_{#3}\prangle}

\newcommand{\nfratsm}[2]{#1\plangle #2\prangle}

\newcommand{\MM}[3]{M_{#1}(#2,#3)}
\newcommand{\mm}[3]{m_{#1}(#2,#3)}

\newcommand{\XXX}[3]{\tilde{X}(#1,#2,#3)}
\newcommand{\xx}[3]{\xi_#1(#2,#3)}

\newcommand{\LLL}[2]{\tilde L_{#1}(#2)}
\newcommand{\lLL}[3]{\mu_{#1}(#2,#3)}
\newcommand{\rRR}[3]{\nu_{#1}(#2,#3)}
\newcommand{\RRR}[2]{\tilde R_{#1}(#2)}

\newcommand{\YYY}[3]{\tilde{Y}(#1,#2,#3)}

\title{Periodic automorphisms of free groups are diagonalisable in free skew-fields}

\author[G. Podlogar]{Gregor Podlogar}
\address{Gregor Podlogar, 
	University of Ljubljana, Faculty of Mathematics and Physics, Ljubljana,
	Slovenia}
\email{gregor.podlogar@fmf.uni-lj.si}

\date{\today}



\definecolor{coolblack}{rgb}{0.0, 0.18, 0.39}

\begin{document}

\begin{abstract}
We study so called weakly-periodic twisted-multiplicative automorphisms of the free skew-field.
In particular, we show that any automorphism of a free skew-field that is defined by a periodic automorphism of a free group is equivalent to a diagonal automorphism.  
\end{abstract}

\maketitle

\tableofcontents

\section{Introduction}

The free skew-field $\nfrats{\C}{x}{d}$ also named the skew-field of noncommutative rational functions is the universal skew-field of quotients of the free algebra $\nfpolys{\C}{x}{d}$ and of the group algebra of the free group $\Gamma$ generated by $x_1,\dots, x_d$.  
Thus, we use notation $\nfrats{\C}{x}{d}= \nfratsm{\C}{\Gamma}$. 

Noncommutative rational functions appear in many areas of mathematics such as control theory \cite{BGM05}, automata theory \cite{ber11}, free probability \cite{hel18} and free real algebraic geometry \cite{hel06, hel12}
In this paper, we are interested in the automorphisms of the free skew-field. In particular, in the so-called multiplicative automorphisms.
We are motivated by the study of the noncommutative rational invariants and a version of the Noncommutative Noethers problem \cite{kl20,pod21}.

Every periodic automorphism of a finitely dimensional vector space is equivalent to a diagonal automorphism if we suitably extend the scalars. We pursue a distant analogy: is every periodic automorphism of free-skew field equivalent to a diagonal automorphism?   
We answer this question for the automorphism of the free skew-fields that are defined by a periodic automorphism of the underling free group. 
The periodic automorphisms of free groups are classified \cite{dy75, mcc80}.

First we  introduce some terminology for the automorphisms of the free skew-field $\nfratsm{\C}{\Gamma}$.
\begin{enumerate}[(1)]
	\item An automorphism  $\sigma$ is \emph{linear} if for every variable $x_i$ the image  $\sigma(x_i)$ is a linear combination of the variables, i.e.,  $\sigma(x_j)\in V=\Lin_\FF\{x_1,\dots,x_d\}$.

	\item An automorphism  $\sigma$ is \emph{diagonal} if for every variable $x_j$ we have $\sigma(x_j)=c_jx_j$ for some $c_j\in \FF$.

	\item 
	An automorphism  $\sigma$ is \emph{multiplicative} if for every variable $x_j$ we have $\sigma(x_j)\in \Gamma$.

	\item 	An automorphism  $\sigma$ is \emph{twisted-multiplicative} if for every variable $x_j$ we have $c_j\sigma(x_j)\in \Gamma$ for some $c_j\in \FF$.
\end{enumerate}
Two automorphisms $\sigma_1$ and $\sigma_2$ are \emph{equivalent} if there exists an automorphism $\zeta$ such that $\zeta\sigma_1\zeta\inv= \sigma_2$.  
 An automorphism is \emph{diagonalisable} if it is equivalent to a diagonal automorphism.

Every multiplicative automorphism is defined by an automorphism of the free group $\Gamma$.
Every twisted-multiplicative automorphism is a composition of a multiplicative and a diagonal automorphism.   
We call a twisted-multiplicative automorphism \emph{weakly-periodic} if it is a composition of a periodic multiplicative automorphism and a diagonal automorphism.

\begin{theorem}\label{thm:main}
	Any weakly-periodic twisted-multiplicative automorphism of $\Aut(\nfratsm{\C}{\Gamma})$ is equivalent to  a diagonal automorphism.
\end{theorem}

\begin{example}\label{ex-1}
The most basic example is the diagonalisation of the automorphism of rational functions in one variable $\nfratsm{\C}{x}=\C(x)$ of order $2$ given by 
$x\mapsto cx\inv$ for any nonzero scalar $c\in C$. 
It is diagonalised in a new variable $y=(\sqrt{c}+x)\inv(\sqrt{c}-x)$, where $\sqrt{c}$ is any solution of the equation $z^2=c$.

A more interesting example is the diagonalisation of the automorphism of $\nfratsm{\C}{x_1, x_2}$ of order $3$ given by
$$
x_0\mapsto c x_1,\quad x_1\mapsto d (x_0x_1)\inv,
$$
where $c,d\in \C$ are any nonzero scalars. 
Let $\omega = \frac{-1 + \sqrt{3}i}{2}$ be the primitive third root of unity. 
The automorphism is linearised in new variables
$$
y_1=\left(1 + c^{-\frac{2}{3}}d^{-\frac{1}{3}} x_0 + c^{-\frac{1}{3}}d^{-\frac{2}{3}}x_0x_1\right)\inv
\left(1 + \omega c^{-\frac{2}{3}}d^{-\frac{1}{3}} x_0 + \omega^2 c^{-\frac{1}{3}}d^{-\frac{2}{3}}x_0x_1\right),
$$
$$
y_2=\left(1 + c^{-\frac{2}{3}}d^{-\frac{1}{3}} x_0 + c^{-\frac{1}{3}}d^{-\frac{2}{3}}x_0x_1\right)\inv
\left(1 + \omega^2 c^{-\frac{2}{3}}d^{-\frac{1}{3}} x_0 + \omega c^{-\frac{1}{3}}d^{-\frac{2}{3}}x_0x_1\right).
$$
We get
$$
y_1 \mapsto \omega^2 y_1, \quad y_2\mapsto \omega y_2.
$$
We sketch that $y_1$ and $y_2$ are generators by expressing $x_0$ and $x_1$ with them. We note that $\omega +\omega^2 =-1$. 
We get 
$$
z_0 = 2(y_1+y_2 +1)\inv = 1 + c^{-\frac{2}{3}}d^{-\frac{1}{3}} x_0 + c^{-\frac{1}{3}}d^{-\frac{2}{3}}x_0x_1
$$
Then we have $$
2 - z_0(y_1+y_2) =  c^{-\frac{2}{3}}d^{-\frac{1}{3}} x_0 + c^{-\frac{1}{3}}d^{-\frac{2}{3}}x_0x_1
$$
and
$$
-i\frac{\sqrt{3}}{3}z_0(y_1-y_2) = c^{-\frac{2}{3}}d^{-\frac{1}{3}} x_0 - c^{-\frac{1}{3}}d^{-\frac{2}{3}}x_0x_1.
$$
From here on, it is clear how to express $x_0$ and $x_1$.
\end{example}

\subsection{Paper's outline}
In Section 2, we provide preliminaries on free skew-fields and periodic automorphisms of free groups.
In Section 3, we prove Theorem~\ref{thm:main}.
In Section 4, we give some corollaries on noncommutative rational invariants of finite cyclic groups.

\section{Preliminaries}

\subsection{Noncommutative rational functions}

We introduce terminology and basic concepts surrounding noncommutative rational functions. 
We work with the base field of complex numbers $\C$. 
For a longer exposition, we refer to \cite{ami66,lew74,Coh95,Coh06,kal12,Vol18}. 

A \emph{noncommutative rational expression} is a syntactically valid combination of scalars for $\C$, variables, operations $+,\cdot,$ inverse and parenthesis, for example:
$$\big(2x_1^3x_2^4x_1^5-((x_1x_2 - x_2x_1)^{-1} + 1)^2\big)^{-1}+x_3.$$ 
Such expressions can be evaluated on tuples of square matrices of equal size with coefficients in $\C$. 
An expression is \emph{nondegenerate} if it is valid to evaluate it on some tuple of matrices. 
Two nondegenerate expressions are equivalent if they evaluate equally whenever both are defined. 
A \emph{noncommutative rational function} is an equivalence class of  a nondegenerate rational expression;
these functions form the free skew-field $\nfrats{\C}{x}{d}$. 
The free skew field $\nfrats{\C}{x}{d}$  is the universal skew-field of fractions of the free algebra  $\nfpolys{\C}{x}{d}$. 
It is universal in the sense that any epimorphism from $\nfpolys{\C}{x}{d}$ to a skew-field $D$ extends to a specialization from  $\nfrats{\C}{x}{d}$ to $D$.

We say that a skew-field is \emph{rational} (or \emph{free}) over $\C$ if it is isomorphic to the free skew-field $\nfrats{\C}{x}{d}$ for some $d\in \N$.
Variables $x_1,\dots, x_d$ in $\nfrats{\C}{x}{d}$ generate a free group $\Gamma$ under multiplication. 
The skew-field $\nfrats{\C}{x}{d}$ is also the universal field of fractions of the group algebra $\C\Gamma$, hence we also use notation $\nfrats{\C}{x}{d}=\nfratsm{\C}{\Gamma}$.

Elements $y_1,\dots,y_d$ of  a skew-field are \emph{rationally independent} over $\FF$ if the skew-field over $\FF$ generated by  $y_1,\dots,y_d$ is isomorphic to $\nfrats{\C}{x}{d}$. 
Rationally independent elements are algebraically independent, but the converse does not hold. 
A set of rationally independent elements that generates $\nfrats{\C}{x}{d}$ is called a \emph{free generating set} and its elements are \emph{free generators}.

\begin{theorem}\label{thm-cohnum}{\cite[Corollary~5.8.14]{Coh95}}
	Any generating set $y_1,\dots, y_d$ of $d$ elements of the free skew-field $\nfrats{\C}{x}{d}$ is a free generating set and any two free generating sets of a free skew-field have the same number of elements.
\end{theorem}

By the above theorem, any automorphism is defined by a mapping from one free generating set to another.

\subsection{Periodic automorphism of free groups and $n,n_1$-trees}

We follow \cite{mcc80}, where the $n,n_1$-trees are used to classify periodic automorphism of a free group. 

Let $\mathcal{T}$ be a rooted tree with a root $a$. 
The other vertices will be denoted by lowercase letters $b,c,d,\dots$. 
The \emph{level} $l(c)$ of a vertex $c$ is the length of the unique reduced path from the root $a$ to $c$. 
The \emph{predecessor} $p(d)$ of a non-root vertex $d$ is the vertex adjacent to $d$ with $l(p(d))=l(d)-1$.
We fix an integer $n\geq 2$ and make the following assignments:
\begin{enumerate}[(1)]
	\item 
	To each vertex $b$ we assign positive integers $\gamma_b$ and $\tau_b$ with $\gamma_b>1$ and $\gamma_b\tau_b=n$. 
	To the root vertex $a$ we assign $\tau_a=1$ and  $\gamma_a=n$.
	\item 
	To a non-root vertex $d$ and its predecessor $p(d)=c$  we assign  integers $\alpha_{c,d},\beta_d$ satisfying $\alpha_{c,d}>1,\beta_d>1, \alpha_{c,d}|\gamma_c,\beta_d|\gamma_d$ and $\gamma_c\beta_d=\al_{c,d}\gamma_d.$
	\item 
	To each ordered pair $(e,f)$ of vertices satisfying $l(e)\leq l(f)$ we assign a  family of positive integers $\delta_{e,f,i},\rho_{e,f,i}, \eta_{e,f,i}$ 
	indexed by a (possibly empty) set $I_{e,f}$, 
	satisfying $\delta_{e,f,i}|\gamma_e,\rho_{e,f,i}|\gamma_f$, $\gamma_e\rho_{e,f,i}=\gamma_f\delta_{e,f,i}$ and $0\leq  \eta_{e,f,i} < \gcd(\tau_e,\tau_f)$.
	If $e\neq f$, we additionally require $\delta_{e,f,i}>1$ and $\rho_{e,f,i}>1$.   
\end{enumerate}
Let $n_1$ be the least common multiple of the set of all $\beta_d\tau_d, \delta_{e,f,i}\tau_e$. 
The tree $\mathcal{T}$ with the assigned integers is called a \emph{(based) $n,n_1$-tree}.

Let  $\mathcal{T}$ be a $n,n_1$-tree. 
For a non-root vertex $d$ of  $\mathcal{T}$, we define a set of letters (variables)
$$
S(d)=\left\{D_0,D_1,\dots, D_{(\beta_d-1)\tau_d-1}\right\},
$$
where the indexed capital letters correspond to the same lowercase letter, e.g., $S(c)=\{C_0,C_1,\dots, C_{(\beta_c-1)\tau_c-1}\}$. 
For a pair of vertices $e,f$ satisfying $l(e)\leq l(f)$  and an index $i\in I_{e,f}$, we define 
$$
S(e,f,i)=\left\{ 
T_{e,f,i,0}, T_{e,f,i,1},\dots, T_{e,f,i,\rho_{e,f,i}\tau_f-1}
\right\}.
$$

Let $S=S(\mathcal{T})$ be a disjoint union of sets $S(d)$ and $S(e,f,i)$,
where $d$ ranges over all non-root vertices of $\mathcal{T}$  and $e,f,i$ range over all triplets  of vertices  $e,f$ with  $l(e)\leq l(f)$  and an index $i\in I_{e,f}$. 
Let $\Gamma = \Gamma_{\mathcal{T}}$ be the free group generated by $S$.

For a vertex $d$ and integers $j\geq 0, k>0$, we denote
$$
\MM{d}{j}{k}=D_jD_{j+\tau_d}D_{j+2\tau_d}\cdots D_{j+(k-1)\tau_d}
$$
and extend the notation to $\MM{d}{j}{0}=1$ (same letters typed differently are in correspondence, for example,  $\MM{c}{1}{3}=C_1C_{1+\tau_c}C_{1+2\tau_c}$).

Next, we define an  endomorphism $\phi=\phi_{\mathcal{T}}$ of $\Gamma$.
For every non-root vertex $d$ of $\mathcal{T}$,
we define 
$$\phi D_j= D_{j+1} \quad \text{for} \quad 0\leq j \leq (\beta_d-1)\tau_d-2.$$ 

Let $A_j=1$ for every non-negative integer $j$. 
We continue inductively. 
We take a non-root vertex  $d$ and $c=p(d)$.
Suppose $\phi^j C_0$ is already defined for every $j\in \N_0$ and denote 
$C_j=\phi^jC_0$. Then we set
$$
\phi D_{(\beta_d-1)\tau_d-1} = \MM{d}{0}{\beta_d-1}\inv \MM{c}{0}{\alpha_{c,d}},
$$
which defines $D_j=\phi^jD_0$ for every $j\in \N$. 
For every ordered pair of vertices  $(e,f)$ with $l(e)\leq l(f)$  and every index $i\in I_{e,f}$, we set
$$\phi T_{e,f,i,j}=T_{e,f,i,j+1}\quad \text{for} \quad 0\leq j \leq \rho_{e,f,i}\tau_f-2$$ and
$$
\phi T_{e,f,i,\rho_{e,f,i}\tau_f-1}=\MM{f}{0}{\rho_{e,f,i}}\inv T_{e,f,i,0}\MM{e}{\eta_{e,f,i}}{\delta_{e,f,i}}.
$$

\begin{theorem}
	For every $n,n_1$-tree $\mathcal{T}$, the endomorphism  $\phi_{\mathcal{T}}$ is an automorphism of  $\Gamma_{\mathcal{T}}$ of order $n_1$.
	
	If $\phi$ is a nontrivial automorphism of order $n$ of a free group $\Gamma$, then there exists a $n,n$-tree $\mathcal{T}$ and an isomorphism 
	$\psi\colon \Gamma \to \Gamma_{\mathcal{T}}$, such that  and $\phi =\psi\inv\phi_{\mathcal{T}}\psi$.
	\end{theorem}

By the above theorem, it is enough to treat the automorphism that are define in terms of an $n,n$-tree.

\begin{example}
The automorphism of the free group generated by the letter $B_0$ defined by
$B_0 \mapsto B_0\inv$ corresponds to a $2,2$-tree with two vertices the root $a$ and a vertex $b$, and $\beta_b=2, \tau_b =1$.

The automorphism of the free group generated by the letters $B_0$ and $B_1$ defined by $B_0\mapsto B_1 \mapsto (B_0B_1)\inv$ corresponds to a $3,3$-tree with two vertices the root $a$ and a vertex $b$, and $\beta_b=3, \tau_b =1$.
\end{example}

\section{Proof of Theorem~\ref{thm:main}}

First, we examine some of the properties of the automorphism defined by a $n,n$-tree $\mathcal{T}$. Then we construct the elements that diagonalise the automorphism.

\begin{lemma}\label{lem-nek2} Let $d$ be a non-root vertex with predecessor $c$.
	\begin{enumerate}[(1)]
		\item $\alpha_{c,d}\tau_c=\beta_d\tau_d$
		\item For any non-negative integer $j$, we have $$D_{(\beta_d+j)\tau_d}=\MM{d}{0}{\beta_d+j}\inv\MM{d}{0}{j+1}\MM{c}{(j+1)\tau_d}{\alpha_{c,d}}.$$
		\item For any non-negative integer $j$, we have 
		\begin{equation*}
			\MM{d}{0}{\beta_d+j}=\MM{d}{0}{j}\MM{c}{j\tau_d}{\alpha_{c,d}}.
		\end{equation*}
		
		\item For any non-negative integers $i,j$ and $k$, we have 
		\begin{equation*}
			\MM{d}{i}{k\beta_d+j}=\MM{d}i{j}\MM{c}{j\tau_d+i}{k\alpha_{c,d}}.
		\end{equation*}
		
		\item There exists $\la\in \N$ such that $\MM{d}0\la=1$.
		Let $\la_d \in \N$  the smallest such that  $\MM{d}0{\la_d}=1$, then we have $\beta_d|\la_d$ and $\la_d|\gamma_d$.
		Furthermore,  $\MM{d}j\la=1$  if and only if $\la$ is divisible by $\la_d$. 
		
		\item Write $\la_d = u_d\beta_d$, let $c = p(d)$ be the predecessor and $\la_c$ the smallest positive integer such that $\MM{c}0{\la_c}=1$. 
		Then
		$u_d\alpha_{c,d}$ is the least common multiple of $\alpha_{c,d}$ and $\la_d$.
	\end{enumerate}
\end{lemma}
\begin{proof}
	(1) From $\gamma_c\tau_c=\gamma_d\tau_d=n$ and $\gamma_c\beta_d=\alpha_{c,d}\gamma_d$ it follows that 
	$$
	\alpha_{c,d}(\gamma_c\tau_c)=(\alpha_{c,d}\gamma_d)\tau_d=\gamma_c\beta_d\tau_d,
	$$
	hence $\alpha_{c,d}\tau_c=\beta_d\tau_d$. 
	
	(2) We use induction on $j$. 
	For the base case $j=0$, we have 
	$$
	D_{\beta_d\tau_d}=\phi^{\tau_d}D_{(\beta_d-1)\tau_d}=\phi^{\tau_d}\left(\MM{d}{0}{\beta_d-1}\inv \MM{d}{0}{\alpha_{c,d}}\right)=\MM{d}{\tau_d}{\beta_d-1}\inv \MM{c}{\tau_d}{\alpha_{c,d}}
	$$
	and $\MM{d}{\tau_d}{\beta_d-1}=D_0\inv \MM{d}{0}{\beta_d}=\MM{d}{0}{1}\inv\MM{d}{0}{\beta_d}$, hence 
	$$D_{\beta_d\tau_d}=\MM{d}{0}{\beta_d}\inv\MM{d}{0}{1}\MM{d}{\tau_d}{\alpha_{c,d}}.$$
	Going from $j$ to $j+1$ we have
	$$
	D_{(\beta_d+j+1)\tau_d}=\phi^{\tau_d}(D_{(\beta_d+j)\tau_d})=
	\phi^{\tau_d}\left(\MM{d}{0}{\beta_d+j}\inv\MM{d}0{j+1}\MM{c}{(j+1)\tau_d}{\alpha_{c,d}}\right)=
	$$
	$$
	\MM{d}{\tau_d}{\beta_d+j}\inv\MM{d}{\tau_d}{j+1}\MM{c}{(j+2)\tau_d}{\alpha_{c,d}}.
	$$
	Applying $\MM{d}{\tau_d}{\beta_d+j}=D_0\inv \MM{d}{0}{\beta_d+j+1}$ and $D_0\MM{d}{\tau_d}{j+1}=\MM{d}{0}{j+1}$, we get 
	$$D_{(\beta_d+j+1)\tau_d}=\MM{d}0{\beta_d+j+1}\inv\MM{d}0{j+2}\MM{c}{(j+2)\tau_d}{\alpha_{c,d}}.$$
	
	(3) We again use induction. 
	For  $j=0$, we get
	$$
	\MM{d}0{\beta_d}=\MM{d}{0}{\beta_d-1}D_{(\beta_d-1)\tau_d}=\MM{d}0{\beta_d-1}\MM{d}0{\beta_d-1}\inv \MM{c}{0}{\alpha_{c,d}}=\MM{c}0{\alpha_{c,d}}.
	$$
	For $j>0$,  we have 
	$$
	\MM{d}0{\beta_d+j}=	\MM{d}0{\beta_d+j-1}D_{(\beta_d+j-1)\tau_d}.
	$$
	Applying  (2), we get $\MM{d}{0}{\beta_d+j}=\MM{d}0{j}\MM{c}{j\tau_d}{\alpha_{c,d}}.$
	
	(4) It is enough to show the equality for $i=0$. 
	For other values, we just apply $\phi^i$.
	We use induction on $k$. The case $k=0$ is trivial and the case $k=1$ is covered in  (3). 
	Going from $k$ to $k+1$ using  (3) and induction hypothesis we get 
	$$
	\MM{d}0{(k+1)\beta_d+j}=\MM{d}0{k\beta_d+j}\MM{c}{k\beta_d\tau_d+j\tau_d}{\alpha_{c,d}}=
	$$
	$$
	=\MM{d}0{j}\MM{c}{j\tau_d}{k\al_{c,d}}\MM{c}{k\beta_d\tau_d+j\tau_d}{\alpha_{c,d}}.
	$$
	Applying $\beta_d\tau_d=\al_{c,d}\tau_c$ and 
	$
	\MM{c}{j\tau_d}{k\al_{c,d}}\MM{c}{k\al_{c,d}\tau_c+j\tau_d}{\alpha_{c,d}}=
	\MM{c}{j\tau_d}{(k+1)\al_{c,d}}
	$ we get
	$$\MM{d}0{(k+1)\beta_d+j}=\MM{d}0{j}\MM{c}{j\tau_d}{(k+1)\al_{c,d}}.$$

	(5) and (6)	Suppose $\mathcal{D}_0(\la)=1$ and write $\la=k\beta_d + j$ for some $k\geq 0$ and $\beta_d>j\geq0$. 
	By  (4) we get
	$$
	\MM{d}0{\la}=\MM{d}0j\MM{c}{j\tau_d}{k\al_{c,d}}.
	$$
	For  $\beta_d>j>0$,  none of the variables appearing in $\MM{c}{j\tau_d}{k\al_{c,d}}$ appear in $\MM{d}0j$ and clearly $\MM{d}0j\neq 1$, hence,  $\beta_d|\la$.
	
	Assuming the existence of $\la_d$, we show that  $\MM{d}0\la=1$ if and only if $\la$ is divisible by $\la_d$.
	First assume $\mathcal{D}_0(\la)=1$.
	Both $\la$ and $\la_d$ are of the form $n\beta_d$ and $m\beta_d$. 
	Write $n=sm+r$ with $0\leq r < m$, then
	$$
	1=\MM{d}0{n\beta_d}=
	\MM{d}0{sm\beta_d+r\beta_d}=
	$$
	$$
	\MM{d}0{m\beta_d}\MM{d}{m\beta_d\tau_d}{m\beta_d}\cdots \MM{d}{(s-1)m\beta_d\tau_d}{m\beta_d}\MM{d}{sm\beta_d\tau_d}{r\beta_d}=\MM{d}{sm\beta_d\tau_d}{r\beta_d}
	$$ 
	and the minimality of $m$ implies $r=0$. 
	The same calculation shows that $\MM{d}0{s\la_d}=1$ for every $s \in \N$.
	
	Finally, we show the existence of $\la_d$.
	We use induction on $l(d)$.
	Denote $p(d)=c$.
	
	For $l(d)=1$, we have $c=a$. Then $C_j=A_j=1$ for every $j$, hence $\MM{d}0{\beta_d}=\MM{a}0{\alpha_{c,d}}=1$. 
	Therefore $\la_d=\beta_d$ and $\la_d|\gamma_d$.
	
	For $l(d)>1$,
	let $\mu$ be the least common multiple of $\la_c$ and $\al_{c,d}$ and write $\mu=v\la_c=u\al_{c,d}$. 
	Using (4), we have $$\MM{d}0{u\beta_d}= \MM{c}{0}{u\al_{c,d}}=\MM{c}{0}{v\la_{c}}=1.$$
	
	In fact, $\la_d=u\beta_d$.
	Suppose otherwise, let $1=\MM{d}0{k\beta_d}$ for $k$ smaller than $u$, then $k\al_{c,d}$ is divisible by $\la_c$  which contradicts the minimality of $\mu$.

	Multiplying $\gamma_c\beta_d=\alpha_{c,d}\gamma_d$ by $u$, we get
	$
	\gamma_c\la_d=\mu\gamma_d.
	$
	Since $\la_c$ and $a_{c,d}$ both divide $\gamma_c$, their least common multiple $\mu$ divides $\gamma_c$ as well, therefore $\la_d|\gamma_d$.
\end{proof}

\begin{remark}
	If  the equality $\MM{d}j{\la} = 1$ holds for some nonnegative integer $j$, then it holds for all nonnegative integers, as 
	$\MM{d}{j'}{\la} = \phi^{j'-j} (\MM{d}{j}{\la}) =  \phi^{j'-j}(1) = 1$.
\end{remark}

Suppose $\tphi = \phi \varphi$ is a weakly-periodic twisted-multiplicative automorphism of $\nfratsm{\C}{\Gamma}$, where $\phi$ is a periodic automorphism of  $\nfratsm{\C}{\Gamma}$ defined by a $n,n$-tree $\mathcal{T}$.
For each vertex $d$ of $\mathcal{T}$ and each integer $j\geq 0$, we have  $\tphi(D_j)=d_{j+1}D_{j+1}$ for some scalar  $d_{j+1}\in \C$.
For integers $j,k$, we denote 
$$
m_d(j,k)=\prod_{l=0}^{k-1} d_{j + l\tau_d}
$$
and extend the notation to $m_d(j,0)=1$.
Then  $\tphi(D_j)^k=\left(\prod_{l=1}^{k} d_{j+l}\right) D_{j+k}$ and $\tphi\left(M_d(j,k)\right) = m_d(j+1,k)M_d(j+1,k).$

%
%
%

We begin the construction of elements of $\nfratsm{\C}{\Gamma}$ that diagonalise $\tphi$.
Pick a non-root vertex $d$ of the tree $\mathcal{T}$ and let $\omega$ be a primitive $\beta_d$-th root of unity.
\begin{proposition}\label{prop-skew1}
	There exist nonzero scalars $\xi_d$ and $\xx{d}{i}{k}$ for $k=0,\dots,\la_d-1$ and $i=0,\dots,\tau_d-1$ such that for

	$$\XXX{d}{i}{j}=\sum_{k=0}^{\la_d-1}\omega^{jk}\xx{d}{i}{k}\MM dik$$ the following equations hold:
	
	\begin{enumerate}
		\item $\tphi\left( \XXX{d}{i}{j} \right) = \XXX{d}{i+1}{j}$ for $i=0,\dots,\tau_d-1$;
		\item $\tphi\left( \XXX{d}{\tau_d-1}{j} \right) = \omega^{-j}\xi_dD_0\inv\XXX{d}{0}{j}$.
	\end{enumerate}
\end{proposition}
\begin{proof}
	For $i=0,\dots,\tau_d-2$, we have
	$$\tphi\left( \XXX{d}{i}{j} \right)= \sum_{k=0}^{\la_d-1}\omega^{jk}\xx{d}{i}{k}\mm{d}{i+1}{k}\MM d{i+1}k.$$
	We get equations $\xx{d}{i}{k}\mm{d}{i+1}{k}= \xx{d}{i+1}{k}$ for $i=0,\dots,\tau_d-2$ that translate to
	\begin{equation}\label{eq:prva}
		\xx{d}{i}{k} = \xx{d}{0}{k} \prod_{l=1}^{i}\mm{d}{l}{k}
	\end{equation} 
	for $i=0,\dots,\tau_d-2$.
	
	We have  
	$$
	\tphi\left( \XXX{d}{\tau_d-1}{j} \right) 
	= 
	\sum_{k=0}^{\la_d-1}\omega^{jk}\xx{d}{\tau_d-1}{k}\mm d{\tau_d}k\MM d{\tau_d}k
	$$
	We equate the above expression to $\omega^{-j}\xi_dD_0\inv\XXX{d}{0}{j}$ and note the equality $\MM d{\tau_d}k=D_0\inv\MM d{0}{k+1}$.
	Comparing the coefficients at $D_0\inv\MM d{0}{k}$  for $k=1,\dots,\la_d-1$, we get 
	$$
	\omega^{-j}\xi_d\omega^{jk}\xx{d}{0}{k} = \omega^{(k-1)j}\xx{d}{\tau_d-1}{k-1}\mm{d}{\tau_d}{k-1}.
	$$
	We apply \eqref{eq:prva} and get 
	$$
	\xx{d}{0}{k} =  \xi_d\inv \xx{d}{0}{k-1}\prod_{l=1}^{\tau_d}\mm{d}{l}{k-1}
	$$
	or equivalently 
	\begin{equation}\label{eq:druga}
		\xx{d}{0}{k} =  \xi_d^{-k} \xx{d}{0}{0}\prod_{l=1}^{\tau_d}\prod_{t=0}^{k-1}\mm{d}{l}{t}
	\end{equation}
	for $k=1,\dots,\la_d-1$.
	
	By equating coefficients at  $D_0\inv\MM d{0}{0}= D_0\inv\MM{d}{0}{\la_d}$, we get
	$$\omega^{-j}\xi_d\xx{d}{0}{0} = \omega^{(\la_d-1)j}\xx{d}{\tau_d-1}{\la_d-1}\mm{d}{\tau_d}{\la_d-1}.$$
	Since $\beta_d |\la_d$, we have $\omega^{(\la_d-1)j}= \omega^{-j}$.
	Applying \eqref{eq:druga}, we get 
	$$
	\xi_d \xx{d}{0}{0} = \xi_d^{-\la_d+1} \xx{d}{0}{0} \prod_{l=1}^{\tau_d}\prod_{t=0}^{\la_d-1}\mm{d}{l}{t}.
	$$
	Finally, we set $\xx{d}{0}{0}=1$ and $\xi_d$ to be any $\la_d$-th root of  $\prod_{l=1}^{\tau_d}\prod_{t=0}^{\la_d-1}\mm{d}{l}{t}$. This defines the scalars such that the desired equations are satisfied.
\end{proof}

The next lemma shows that the elements 
$$\XXX{d}{i}{j}=\sum_{k=0}^{\la_d-1}\omega^{jk}\xx{d}{i}{k}\MM dik$$
are non-zero, hence they have an inverse.
\begin{lemma}
	The elements $\MM dik$, $k=0,\dots,\la_d-1$ are distinct.
\end{lemma}
\begin{proof}
Suppose	$\MM di{k_1}=\MM di{k_2}$ for $0\leq k_1 < k_2 \leq \la_d-1$. 
By definition, we have
$\MM di{k_2} =\MM di{k_1}\MM d{i+k_1\tau_d}{k_2-k_1}$, hence $\MM d{i+k_1\tau_d}{k_2-k_1}=1$. 
This contradicts the minimality of $\la_d$.
	\end{proof}

We fix the scalars $\xi_d$ and $\xx{d}{i}{k}$ from Proposition~\ref{prop-skew1} and denote 
$$\YYY{d}{i}{j} = \XXX{d}{i}{0}\inv \XXX{d}{i}{j}$$
for $i=0,\dots, \tau_d-1$ and $j=1,\dots, \beta_d-1$. The inverse $\XXX{d}{i}{0}\inv$ exists by the previous lemma.

\begin{lemma} \phantomsection \label{lem-lin1}
	\begin{enumerate}[(1)]
		\item $\tphi\left(\YYY{d}{i}{j}\right)= \tphi\left(\YYY{d}{i+1}{j}\right)$ for  $i=0,\dots, \tau_d-2$,
		\item $\tphi\left(\YYY{d}{\tau_d-1}{j}\right)= \omega^{-j}\tphi\left(\YYY{d}{0}{j}\right).$
	\end{enumerate}
\end{lemma}
\begin{proof}	
	Follows from Proposition~\ref{prop-skew1}.
\end{proof}
\begin{lemma}\label{lem-navad}
	Write $\la_d=u_d\beta_d$.
	\begin{enumerate}[(1)]
		\item  $\sum_{j=1}^{\beta_d-1}  \XXX{d}{i}{j}=\beta_d\sum_{s=0}^{u_d-1}\xx{d}{i}{s\beta_d}\MM{c}{i}{s\alpha_{c,d}} -\MM{d}{i}{0}$.
		\item $\sum_{j=1}^{\beta_d-1} \YYY{d}{i}{j}=\beta_d\XXX{d}{i}{0}\inv\left(\sum_{s=0}^{u_d-1}\xx{d}{i}{s\beta_d}\MM{c}{i}{s\alpha_{c,d}}\right)-1$.
		\item For $s=0,\dots,u_d-1$, the elements are distinct $\MM{c}{i}{s\alpha_{c,d}}$
	\end{enumerate}
\end{lemma}
\begin{proof}	
	(1) For $\la_d=u\beta_d$, we have 
	$$
	\sum_{j=1}^{\beta_d-1} \XXX{d}{i}{j}
	=
	\sum_{k=0}^{\la_d-1} \left(\sum_{j=1}^{\beta_d-1}\omega^{jk}\right) \
	= 	
	\sum_{s=0}^{u_d-1}\sum_{k=0}^{\beta_d-1} \left(\sum_{j=1}^{\beta_d-1}\omega^{jk}\right)\xx{d}{i}{s\beta_d +k}\MM di{s\beta_d +k}
	$$
	We apply  $$
	\sum_{j=1}^{\beta_d-1}\omega^{jk}=
	\begin{cases}
		\beta_d-1&; \beta_d|k\\
		-1&; \text{otherwise}
	\end{cases}
	$$
	to the above sum and get
	$$
	\sum_{j=1}^{\beta_d-1} \XXX{d}{i}{j}
	=
	(\beta_d-1)\sum_{s=0}^{u_d-1}\xx{d}{i}{s\beta_d}\MM di{s\beta_d}- 	\sum_{s=0}^{u_d-1}\sum_{k=1}^{\beta_d-1}\xx{d}{i}{s\beta_d +k}\MM di{s\beta_d +k}.
	$$
	Reorganizing the sum, we get 
	$$
	\sum_{j=1}^{\beta_d-1}  \XXX{d}{i}{j}
	=
	\beta_d\sum_{s=0}^{u_d-1}\xx{d}{i}{s\beta_d}\MM di{s\beta_d}- 	\sum_{k=0}^{\la_d-1}\xx{d}{i}{k}\MM di{k}.
	$$
	We get the desired equality by applying $\MM di{s\beta_d}= \MM{c}{i}{s\alpha_{c,d}}$ from (4) of Lemma~\ref{lem-nek2}.

	(1) Directly from  (2).
	
	(3) Suppose $\MM{c}{i}{s_1\alpha_{c,d}}= \MM{c}{i}{s_2\alpha_{c,d}}$ for $0\leq s_1<s_2\leq u-1$.
We have $\MM{c}{i}{s_2\alpha_{c,d}} = \MM{c}{i}{s_1\alpha_{c,d}}\MM{c}{i+s_1\tau_c}{(s_2-s_1)\alpha_{c,d}}$, hence 
$\MM{c}{i+s_1\tau_c}{(s_2-s_1)\alpha_{c,d}}=1$. 
Then $(s_2-s_1)\alpha_{c,d}$ is a multiple of $\la_c$, this contradicts  $u_d\alpha_{c,d}$ being the least common multiple of $\la_d$ and $\alpha_{c,d}$ 
\end{proof}

For vertices $e,f$ of $\mathcal{T}$ satisfying $l(e)\leq l(f)$ and $i\in I_{e,f}$, let $u_{e,f,i}$ and $v_{e,f,i}$ be the smallest $u,v\in \N$ such that $\MM{f}{0}{u\rho_{e,f,i}}=1$ and $\MM{e}{0}{v\delta_{e,f,i}}=1$. 
Integer with such property exists by (5) of Lemma~\ref{lem-nek2}.

\begin{lemma}\phantomsection\label{lemma:enako}\begin{enumerate}
		\item 
		$\rho_{e,f,i}\tau_f=\delta_{e,f,i}\tau_e$
		\item For  $k =0, \dots ,u_{e,f,i}-1$, the elements  $\MM{f}j{k\rho_{e,f,i}}$ are distinct.
		\item For  $k =0, \dots ,v_{e,f,i}-1$, the elements  $\MM{e}{j+\eta_{e,f,i}}{k\delta_{e,f,i}}$ are distinct.
	\end{enumerate}
\end{lemma}
\begin{proof}
	(1) We have $\gamma_e\tau_e=\gamma_f\tau_f=n$ and $\gamma_e\rho_{e,f,i}=\gamma_f\delta_{e,f,i}$ by definition.
	Then $$
	\rho_{e,f,i}\gamma_f\tau_f=\rho_{e,f,i}\gamma_e\tau_e=\gamma_f\delta_{e,f,i}\tau_e
	$$
	and $\rho_{e,f,i}\tau_f=\delta_{e,f,i}\tau_e.$
	
	(2) Suppose $\MM{f}j{k_1\rho_{e,f,i}} = \MM{f}j{k_2\rho_{e,f,i}}$ for  $0 \leq k_1 < k_2 < u_{e,f,i}$. 
	Then  $\MM{f}j{k_2\rho_{e,f,i}} = \MM{f}j{k_1\rho_{e,f,i}}\MM{f}{j+ k_1\rho_{e,f,i}\tau_e}{(k_2-k_1)\rho_{e,f,i}} = \MM{f}j{k_1\rho_{e,f,i}}$, hence
	$\MM{f}{j+ k_1\rho_{e,f,i}\tau_e}{(k_2-k_1)\rho_{e,f,i}}= 1$ which contradicts the minimality of $u_{e,f,i}$.
	
	(3) The reasoning is the same as for (2).
\end{proof}

\begin{proposition}\label{prop:stran}
	There exist nonzero scalars $\mu_{e,f,i}$, $\lLL{e,f,i}{j}{k}$ and $\nu_{e,f,i}$, $\rRR{e,f,i}{j}{k}$ for $j=0,1,\dots, u_{e,f,i}-1$ such that for
	$$
	\LLL{e,f,i}{j}=\sum_{k=0}^{u_{e,f,i}-1} \lLL{e,f,i}{j}{k} \MM{f}j{k\rho_{e,f,i}}\inv
	$$
	and 
	$$
	\RRR{e,f,i}{j}=\sum_{k=0}^{v_{e,f,i}-1} \rRR{e,f,i}{j}{k} \MM{e}{j+\eta_{e,f,i}}{k\delta_{e,f,i}}\inv
	$$
	the following equations hold
	\begin{enumerate}
		\item For $j = 0, \dots, \rho_{e,f,i}\tau_f-2$, we have $\tphi\left(\LLL{e,f,i}{j}\right) = \LLL{e,f,i}{j+1}$;
		\item $\tphi\left(\LLL{e,f,i}{\rho_{e,f,i}\tau_f-1}\right) =\mu_{e,f,i}\LLL{e,f,i}{0}\MM{f}{0}{\rho_{e,f,i}}$;
		\item For $j = 0, \dots, \rho_{e,f,i}\tau_f-2$, we have $\tphi\left(\RRR{e,f,i}{j}\right) = \RRR{e,f,i}{j+1}$;
		\item $\tphi\left(\RRR{e,f,i}{\rho_{e,f,i}\tau_f-1}\right) = \nu_{e,f,i}\MM{e}{\eta_{e,f,i}}{\delta_{e,f,i}}\inv\RRR{e,f,i}{0}$.
	\end{enumerate}
\end{proposition}
\begin{proof}
	We first find scalars that satisfy (1) and (2).
	For $j = 0, \dots, \rho_{e,f,i}\tau_f-2$, we get
	$$\tphi\left(\LLL{e,f,i}{j}\right) = 	\sum_{k=0}^{u_{e,f,i}-1} \lLL{e,f,i}{j}{k} \mm{f}{j+1}{k\rho_{e,f,i}}\inv\MM{f}{j+1}{k\rho_{e,f,i}}\inv$$
	that implies equations 
	$
	\lLL{e,f,i}{j}{k} \mm{f}{j+1}{k\rho_{e,f,i}}\inv = \lLL{e,f,i}{j+1}{k} 
	$
	and further
	\begin{equation}\label{eq:faf}
		\lLL{e,f,i}{j}{k} = \lLL{e,f,i}{0}{k} \prod_{l=1}^j \mm{f}{l}{k\rho_{e,f,i}}\inv. 
	\end{equation}
	
	We have 
	$$
	\tphi\left(\LLL{e,f,i}{\rho_{e,f,i}\tau_f-1}\right) = 	\sum_{k=0}^{u_{e,f,i}-1} \lLL{e,f,i}{\rho_{e,f,i}\tau_f-1}{k} \mm{f}{\rho_{e,f,i}\tau_f}{k\rho_{e,f,i}}\inv\MM{f}{\rho_{e,f,i}\tau_f}{k\rho_{e,f,i}}\inv.
	$$
	We apply 
	$$\MM{f}{\rho_{e,f,i}\tau_f}{k\rho_{e,f,i}}\inv = \MM{f}{0}{(k+1)\rho_{e,f,i}}\inv \MM{f}{0}{\rho_{e,f,i}},$$
	$\MM{f}{0}{u_{e,f,i}\rho_{e,f,i}} =  \MM{f}{0}{0}$
	and
	compare the coefficients to get equations
	$$	
	\mu_{e,f,i}\lLL{e,f,i}{0}{k} = \lLL{e,f,i}{\rho_{e,f,i}\tau_f-1}{k-1}\mm{f}{\rho_{e,f,i}\tau_f}{(k-1)\rho_{e,f,i}}\inv
	$$
	for $k=1,\dots, u_{e,f,i}-1$.
	We apply \eqref{eq:faf} and get 
	$$
	\mu_{e,f,i}\lLL{e,f,i}{0}{k} = \lLL{e,f,i}{0}{k-1}\prod_{l=1}^{\rho_{e,f,i}\tau_f}\mm{f}{l}{(k-1)\rho_{e,f,i}}\inv
	$$
	and further
	\begin{equation}\label{eq:faf2}
		\lLL{e,f,i}{0}{k} = \mu_{e,f,i}^{-k} \lLL{e,f,i}{0}{0}\prod_{t=0}^{k-1}\prod_{l=1}^{\rho_{e,f,i}\tau_f}\mm{f}{l}{t\rho_{e,f,i}}\inv.
	\end{equation}
	Comparing the coefficients at $\MM{f}{0}{\rho_{e,f,i}}$, we get equation
	$$
	\mu_{e,f,i}\lLL{e,f,i}{0}{0} = \lLL{e,f,i}{\rho_{e,f,i}\tau_f-1}{u_{e,f,i}\rho_{e,f,i}}\mm{f}{l}{u_{e,f,i}\rho_{e,f,i}}\inv.
	$$
	We apply \eqref{eq:faf} and \eqref{eq:faf2} to get
	$$
	\mu_{e,f,i}^{u_{e,f,i}}\lLL{e,f,i}{0}{0} = \lLL{e,f,i}{0}{0}\prod_{t=0}^{u_{e,f,i}}\prod_{l=1}^{\rho_{e,f,i}\tau_f}\mm{f}{l}{t\rho_{e,f,i}}\inv.
	$$
	We take $\lLL{e,f,i}{0}{0}=1$ and $\mu_{e,f,i}$ to be any $u_{e,f,i}$-th root of $\prod_{l=1}^{\rho_{e,f,i}\tau_f}\mm{f}{l}{t\rho_{e,f,i}}\inv$ to get the desired scalars.
	
	The scalars $\nu_{e,f,i}$, $\rRR{e,f,i}{j}{k}$ that satisfy (3) and (4) are constructed in a very similar manner. 
	We just use $\rho_{e,f,i}\tau_f=\delta_{e,f,i}\tau_e$ from Lemma~\ref{lemma:enako}, 
	$$\MM{e}{\rho_{e,f,i}\tau_f + \eta_{e,f,i}}{k\delta_{e,f,i}} = \MM{e}{\eta_{e,f,i}}{\delta_{e,f,i}}\inv\MM{e}{\eta_{e,f,i}}{(k+1)\delta_{e,f,i}}$$ and 
	$\MM{e}{j}{v_{e,f,i}\delta_{e,f,i}} = 1$.
\end{proof}

We fix the scalars from the above proposition and define
$$
\widetilde{T}_{e,f,i,j}=	\LLL{e,f,i}{j}T_{e,f,i,j}\RRR{e,f,i}{j}
$$
 for $j=0,1,\dots, \rho_{e,f,i}\tau_f-1$. We note that $\LLL{e,f,i}{j}$ and $\RRR{e,f,i}{j}$ are nonzero  by Lemma~\ref{lemma:lin2}.

\begin{lemma}\phantomsection\label{lemma:lin2}
	\begin{enumerate}[(1)]
		\item $\tphi\widetilde{T}_{e,f,i,j} = \tphi\widetilde{T}_{e,f,i,j}$ for $j= 0,1,\dots, \rho_{e,f,i}\tau_f-2$
		\item $\tphi\widetilde{T}_{e,f,i,\rho_{e,f,i}\tau_f-1} = \mu_{e,f,i}\nu_{e,f,i}\widetilde{T}_{e,f,i,0}.$
	\end{enumerate}
\end{lemma}
\begin{proof}
From Proposition~\ref{prop:stran}.
\end{proof}

For a non-root vertex $d$,  we define 
$$
\widetilde{S}(d)= \left\{ \YYY{d}{i}{j} \mid i=0,\dots,\tau_d-1,\ j=1,\dots,\beta_d-1 \right\}
$$ 
and for a  pair $(e,f)$ with $l(e)\leq l(f)$ and $i \in I_{e,f}$ we define
$$
\widetilde{S}(e,f,i)=\left\{ \widetilde{T}_{e,f,i,j}\mid j=0,\dots,\rho_{e,f,i}\tau_f-1 \right\}.
$$
Let $\widetilde{S}$ be the union of all $\widetilde{S}(d)$ and all $\widetilde{S}(e,f,i)$ where $d$ ranges through all non-root vertices of $\mathcal{T}$, $e,f$ range through all  pairs of  vertices with  $l(e)\le l(f)$ and for given vertices $e,f$,  the index $i$ ranges through the index set $I_{e,f}$.

\begin{proposition} 
	The set $\widetilde{S}$ is a free generating set of $\nfratsm{\C}{\Gamma_\mathcal{T}}$.
\end{proposition}
\begin{proof}
	Since $|S|=|\widetilde{S}|$, it is enough to show that $\widetilde{S}$ generates $\nfratsm{\FF}{\Gamma_\mathcal{T}}$ by Theorem~\ref{thm-cohnum}.
	We  use induction on $l(d)$ and for a given vertex $d$,  show that elements of $S(d)$ can be expressed using elements of $\widetilde{S}(d)$ and elements of  the sets $S(c)$ with $l(c)< l(d)$ that can be expressed with the elements of $\widetilde{S}$ by induction hypothesis. 
	We denote the union of these sets by $\overline{S(d)}$.
	
	The base of the induction $l(d)=0$ is trivial. We assume $l(d)>0$.
	By  (4) of Lemma~\ref{lem-navad}, we can express $\XXX{d}{i}{0}$ for $i=0,1,\dots,\tau_d-1$, using elements of $\overline{S(d)}$. 
	Hence, we can also express $\XXX{d}{i}{j}=\XXX{d}{i}{0}\YYY{d}{i}{j}$ for $j=1,\dots,\beta_d-1$ and $i=0,1,\dots,\tau_d-1$.
	Using  (4) of Lemma~\ref{lem-nek2} and $u_d\beta_d=\la_d$, we write 
	$$
	\XXX{d}{i}{j}=\sum_{s=0}^{u_d-1}\sum_{k=0}^{\beta_d-1} \omega^{jk}\xx{d}{i}{s\beta_d+ k}\MM{d}{i}{s\beta_d+ k}=
	$$
	$$
	\sum_{s=0}^{u_d-1}\sum_{k=0}^{\beta_d-1} \omega^{jk}\xx{d}{i}{s\beta_d+ k}\MM{d}{i}{k}\MM{c}{k\tau_d + i}{s\al_{c,d}}=
	$$
	$$
	\sum_{k=0}^{\beta_d-1} \omega^{jk}  \MM{d}{i}{k}\left( \sum_{s=0}^{u_d-1}\xx{d}{i}{s\beta_d+ k} \MM{c}{k\tau_d + i}{s\al_{c,d}} \right).
	$$
	Since the Vandermonde matrix $V(1, \omega, \omega ^2, \dots, \omega ^{\beta_d-1})$ is invertible, we can express $$\MM{d}{i}{k}\left( \sum_{s=0}^{u_d-1}\xx{d}{i}{s\beta_d+ k} \MM{c}{k\tau_d + i}{s\al_{c,d}} \right)$$ as a linear combination of
	$\XXX{d}{i}{j}$ for $j=1,\dots,\beta_d-1$. 
We note that $$\left( \sum_{s=0}^{u_d-1}\xx{d}{i}{s\beta_d+ k} \MM{c}{k\tau_d + i}{s\al_{c,d}} \right)$$ in nonzero by (3) of Lemma~\ref{lem-navad}, thus, we can express $\MM{d}{i}{k}$ for $i=0,1,\dots,\tau_d-1$, $k=1,\dots,\beta_d-1$ using only elements from $\overline{S(d)}$.
	Using these, we can express initial variables  $D_s$ for $s=0,1,\dots, (\beta_d-1)\tau_d-1$ as desired.
	
	For vertices $e,f$ and index $i \in I_{e,f}$, we can express elements of $S(e,f,i)$ using elements from $\widetilde S(e,f,i)$, $S(e)$ and $S(f)$ in a straightforward manner as .
\end{proof}

By Lemmas~\ref{lem-lin1}  and \ref{lemma:lin2}, it is clear that the automorphism $\tphi$ is linear on the set of generators $\widetilde{S}$.
For each vertex $d$, the vector space  $V(d)=\Lin\widetilde{S}(d)$ is closed under $\tphi$ and 
$$
(\tphi|_{V(d)})^{\tau_d} = \omega^{-j} \id
$$
by Proposition~\ref{lem-lin1}, hence there exist a basis of $V(d)$ such that $\tphi|_{V(d)}$ is diagonal on it. 
In the set of generators   $\widetilde{S}$, we swap $\widetilde{S}(d)$ with this basis.

For a pair of vertices $e,f$ with  $l(e)\le l(f)$ and an index $i\in I_{e,f}$, the vector space 
$V(e,f,i)= \Lin\widetilde{S}(e,f,i)$ is closed under $\tphi$ and 
$$
(\tphi|_{V(e,f,i)})^{ \rho_{e,f,i}\tau_f} = \mu_{e,f,i}\nu_{e,f,i} \id
$$
by Proposition~\ref{lemma:lin2}, hence there exist a basis of $V(e,f,i)$ such that $\tphi|_{V(e,f,i)}$ is diagonal on it. 
In the set of generators   $\widetilde{S}$, we swap $\widetilde{S}(e,f,i)$ with this basis.

We have found a set of free generators of $\nfratsm{\C}{\Gamma}$ such that $\tphi$ is diagonal on them, thus finished the proof of Theorem~\ref{thm:main}.

\section{Noncommutative rational invariants of cyclic groups}

Noncommutative rational functions invariant under linear actions of solvable groups are studied in
\cite{kl20,pod21}. The result we use is:
	if a finite abelian group $A$ acts faithfully diagonally on  $\nfrats{\C}{x}{d}$, then 
	the skew-field of invariants $\nfrats{\C}{x}{d}^A$ is free with $|A|(d-1)+1$ free generators \cite[Thm 4.1]{kl20}. 
	
We note that any faithful action of a cyclic group $C_n$ on $\nfrats{\C}{x}{d}$ is defined via an automorphism of order $n$. 

\begin{corollary}
	Let the cyclic group $C_n$ act on $\nfrats{\C}{x}{d}$ via an twisted-multiplicative automorphism, then the skew-field of invariants 
	$\nfrats{\C}{x}{d}^{C_n}$ is free with  $n(d-1)+1$ free generators.
\end{corollary}
\begin{proof}
	Directly from Theorem~\ref{thm:main} and \cite[Thm 4.1]{kl20}.
\end{proof}


\bibliography{references}

\begin{thebibliography}{KPPV20}

\bibitem[Ami66]{ami66}
S.A Amitsur.
\newblock Rational identities and applications to algebra and geometry.
\newblock {\em J. Algebra}, 3(3):304--359, 1966.

\bibitem[BGM05]{BGM05}
Joseph~A. Ball, Gilbert Groenewald, and Tanit Malakorn.
\newblock Structured noncommutative multidimensional linear systems.
\newblock {\em SIAM J. Control Optim.}, 44(4):1474--1528, 2005.

\bibitem[BR11]{ber11}
Jean Berstel and Christophe Reutenauer.
\newblock {\em Noncommutative rational series with applications}, volume 137 of
  {\em Encyclopedia of Mathematics and its Applications}.
\newblock Cambridge University Press, Cambridge, 2011.

\bibitem[Coh95]{Coh95}
Paul~Moritz Cohn.
\newblock {\em Skew Fields: Theory of General Division Rings}.
\newblock Encyclopedia of Mathematics and its Applications. Cambridge
  University Press, Cambridge, 1995.

\bibitem[Coh06]{Coh06}
Paul~Moritz Cohn.
\newblock {\em Free ideal rings and localization in general rings}, volume~3 of
  {\em New Mathematical Monographs}.
\newblock Cambridge University Press, Cambridge, 2006.

\bibitem[DPS75]{dy75}
Joan~L Dyer and G~Peter~Scott.
\newblock Periodic automorphisms of free groups.
\newblock {\em Commun. Algebra}, 3(3):195--201, 1975.

\bibitem[HKM12]{hel12}
J~William Helton, Igor Klep, and Scott McCullough.
\newblock Chapter 8: Free convex algebraic geometry.
\newblock In {\em Semidefinite optimization and convex algebraic geometry},
  pages 341--405. SIAM, 2012.

\bibitem[HMS18]{hel18}
J~William Helton, Tobias Mai, and Roland Speicher.
\newblock Applications of realizations (aka linearizations) to free
  probability.
\newblock {\em J. Funct. Anal.}, 274(1):1--79, 2018.

\bibitem[HMV06]{hel06}
J~William Helton, Scott~A McCullough, and Victor Vinnikov.
\newblock Noncommutative convexity arises from linear matrix inequalities.
\newblock {\em J. Funct. Anal.}, 240(1):105--191, 2006.

\bibitem[KPPV20]{kl20}
Igor Klep, James~Eldred Pascoe, Gregor Podlogar, and Jurij Vol{\v{c}}i{\v{c}}.
\newblock Noncommutative rational functions invariant under the action of a
  finite solvable group.
\newblock {\em J. Math. Anal. Appl}, 490(2):124341, 2020.

\bibitem[KVV12]{kal12}
Dmitry~S. Kaliuzhnyi-Verbovetskyi and Victor Vinnikov.
\newblock Noncommutative rational functions, their difference-differential
  calculus and realizations.
\newblock {\em Multidimens. Syst. Signal Process.}, 23(1-2):49--77, 2012.

\bibitem[Lew74]{lew74}
Jacques Lewin.
\newblock Fields of fractions for group algebras of free groups.
\newblock {\em Trans. Amer. Math. Soc.}, 192:339--346, 1974.

\bibitem[McC80]{mcc80}
James McCool.
\newblock A characterization of periodic automorphisms of a free group.
\newblock {\em Trans. Amer. Math. Soc}, 260(1):309--318, 1980.

\bibitem[Pod23]{pod21}
Gregor Podlogar.
\newblock Finite solvable groups with a rational skew-field of noncommutative
  real rational invariants.
\newblock {\em Commun. Algebra}, 51(6):2268--2292, 2023.

\bibitem[Vol18]{Vol18}
Jurij Vol\v{c}i\v{c}.
\newblock Matrix coefficient realization theory of noncommutative rational
  functions.
\newblock {\em J. Algebra}, 499:397--437, 2018.

\end{thebibliography}

\bibliographystyle{alpha}

\end{document}